\documentclass[letterpaper, 11 pt, onecolumn, doublespacing]{ieeeconf}
\IEEEoverridecommandlockouts                              
\usepackage{amsmath,epsfig, cite, cases}
\usepackage{subfigure,wrapfig}
\usepackage{bbm,amssymb,amscd,mathrsfs}
\usepackage[active]{srcltx}
\usepackage{color}
\usepackage{setspace}
\usepackage{enumerate}
\usepackage{bbding}
\usepackage{algorithmic}
\usepackage{algorithm}
\usepackage[cm]{fullpage}

\newtheorem{thm}{Theorem}[section]

\newtheorem{lem}[thm]{Lemma}

\newtheorem{define}{Definition}

\def\crate{\mu}
\def\cbar{\overline{c}}
\def\Lam{\Lambda}
\def\eps{\epsilon}
\def\L{\mathcal{L}}
\def\E{\mathbb{E}}

\def\F{{\mathbf{F}}}
\def\Re{{\mathbb{R}}}
\def\mE{\mathcal{E}}
\def\mN{\mathcal{N}}
\def\mG{\mathcal{G}}

\def\ones{\mathbf{1}}
\def\ite{\textit{e}}
\def\itw{\textit{w}}
\def\itW{\textit{W}}
\def\mN{\mathcal{N}}
\def\mV{\mathcal{V}}
\def\mT{\mathcal{T}}
\def\mS{\mathcal{S}}

\newcommand{\Qin}[1][t]{Q_{in,#1}}
\newcommand{\Qout}[1][t]{Q_{out,#1}}
\newcommand{\bra}[1]{\left(#1\right)}

\newcommand{\argmax}{\operatornamewithlimits{argmax}}

\newcommand{\xx}[1][i]{x^{#1}}

\newcommand{\indicator}[1]{\mathbbm{1}_{\left[ {#1} \right] }}

\title{Tree Codes Improve Convergence Rate of Consensus Over Erasure Channels}

\author{Ravi Teja Sukhavasi and Babak Hassibi\\California
  Institute of Technology}
\begin{document}

\maketitle
\thispagestyle{empty}
\pagestyle{empty}

\begin{abstract}
We study the problem of achieving average consensus between a group of
agents over a network with erasure links. In the context of consensus
problems, the unreliability of communication links between nodes has been
traditionally modeled by allowing the underlying graph to vary with
time. In other words, depending on the realization of the link
erasures, the underlying graph at each time instant is assumed to be a
subgraph of the original graph. Implicit in this model is the
assumption that the erasures are symmetric: if at time $t$ the packet
from node $i$ to node $j$ is dropped, the same is true for the packet
transmitted from node $j$ to node $i$. However, in practical wireless
communication systems this assumption is unreasonable and, due to the
lack of symmetry, standard averaging protocols cannot guarantee that
the network will reach consensus to the true average. In
this paper we explore the use of channel coding to improve the
performance of consensus algorithms. For symmetric erasures, we show
that, for certain ranges of the system parameters, repetition codes
can speed up the convergence rate. For asymmetric erasures we show
that tree codes (which have recently been designed for erasure
channels) can be used to simulate the performance of the original
``unerased'' graph. Thus, unlike conventional consensus methods, we can
guarantee convergence to the average in the asymmetric case. The price
is a slowdown in the convergence rate, relative to the unerased
network, which is still often faster than the convergence rate of
conventional consensus algorithms over noisy links.
\end{abstract}
\section{Introduction}

In a network of agents, consensus refers to the process of achieving
agreement between the agents in a distributed manner. Consensus problems,
and in particular the problem of reaching consensus on the average of
the values of the agents, 
have been around for a while and are often used to serve as a test
case for studying distributed computation and decision making between
a group of nodes/processors/dynamical systems (\cite{Murray04, Murray07, Lin03, Tsitsiklis84, Tsitsiklis86, Ren05}). Most of the work in this area assumes that the
agents are connected via a fixed underlying graph or network. In many
applications, however, the links in the underlying graph are noisy or
unreliable. In the context of consensus
problems, the unreliability of communication links between nodes has been
traditionally modeled by allowing the underlying graph to vary with
time. In other words, at each time instant some of the links are allowed
to be erased, and depending on the realization of the link
erasures, the underlying graph at each time instant is assumed to be a
subgraph of the original graph. Furthermore, the distributed algorithm
for reaching consensus remains unchanged: the same distributed
averaging algorithm is used, except that only the information received
at each time is used. An important assumption that is implicitly made
in this model is that the erasures are symmetric: if at time $t$ the packet
from node $i$ to node $j$ is dropped, the same is true for the packet
transmitted from node $j$ to node $i$. In practical wireless
communication systems this assumption is patently unreasonable: the
additive noise at the two nodes are independent and, furthermore,
communication in the two directions occurs at either different times
or over different frequency bands. If standard averaging protocols are
performed, this loss of symmetry can prohibit the network from
reaching consensus to the true average (standard consensus protocols
require that the ``update'' matrix be doubly stochastic, something
that cannot be guaranteed in the asymmetric case). 

The goal of this paper is to explore the use of channel coding to improve the
performance of consensus algorithms, especially in the asymmetric
case. A major impetus for this work is the recently designed tree
codes for erasure channels \cite{Sukhavasi1101}, which, as we demonstrate, resolves
the problem encountered in the asymmetric case. 

For asymmetric erasures we show that tree codes can be used
to simulate the performance of the original {\em unerased}
graph. Thus, unlike conventional consensus methods, we can 
{\em guarantee} convergence to the average in the asymmetric case. As
expected, the price is a slowdown in the convergence rate, relative to
the convergence rate of the unerased network. Nonetheless, the
slowdown is still often faster than the convergence rate of
conventional consensus algorithms over erasure links.
\section{Problem Setup}

Consider a group of $N$ nodes denoted by $\mN = \{1,2,\ldots,N\}$. We assume that the nodes are connected by an undirected communication graph $\mG = (\mN,\mE)$ which is often referred to as the interaction graph. Throughout the analysis $\mG$ is assumed to time invariant. Let $A = [a_{ij}]$ denote the adjacency matrix of $G$, i.e., $a_{ij} = 1$ if $(i,j)\in\mE$ and $0$ otherwise. Note that $a_{ii} = 0$. Let $\xx_0$ denote the initial value at node $i$. The objective is for the nodes to compute the global average $r = \frac{1}{N}\ones^Tx_0$, where $\ones$ denotes an $N$-dimensional column of ones and $x_0$ is the column vector of the $\xx_0$'s. We model the communication links between nodes as packet erasure links. Further, we ignore quantization effects due to packetization. The standard packet sizes in practice justify this assumption. We denote the event of successful packet reception from node $j$ to node $i$ at time $k$ with the Bernoulli random variable $X_k^{ij}$, i.e., $X_k^{ij} = 1$ if the packet is received successfully at time $k$ and $0$ otherwise. This notation is summarized in Table \ref{tab: notation}. 

\begin{table}
 \caption{Notation}
 \label{tab: notation}
\begin{center}
\begin{tabular}{|r|l|}
\hline
$\Vert y\Vert$, $y\in\Re^N$ & $\sqrt{\sum_{i=1}^N y_i^2}$, i.e., the two norm of $y$ \\
 $\mN = \{1,2,\ldots,N\}$ & the set of nodes\\
$\mG = (\mN,\mE)$ & the underlying communication graph\\
$A = [a_{ij}]$ & the adjacency matrix of $G$, i.e.,\\
 &  $a_{ij} = 1$ if $(i,j)\in\mE$ and $0$ otherwise\\
$\Delta$ & largest degree of any vertex in $G$\\
$\xx_0$ & the initial value at node $i$\\
$x_0$ & column of $\xx_0$'s\\
$r$ & the initial average, i.e., $\frac{1}{N}\ones^Tx_0$\\
$p$ & packet erasure probability\\
$X_k^{ij}$ & 1 if the packet sent from node $j$ to\\
 & node $i$ at time $k$ is successfully\\
 & received and 0 o.w\\
$\rho(.)$ & spectral radius of a matrix\\
$A\circ B$ & Hadamard product,i.e., \\
& $(A\circ B)_{ij} = A_{ij}B_{ij}$\\
$A\otimes B$ & Kronecker product\\
$D(p,q)$ & $p\log\frac{p}{q} + (1-p)\log\frac{1-p}{1-q}$\\
& i.e., Kullbeck Leibler divergence\\
& between Bernoulli($p$) and \\
& Bernoulli($q$)\\
\hline
\end{tabular}
\end{center}
\end{table}
\section{Background}

For a fixed communication graph $\mG$, a typical algorithm to achieve consensus is of the following form. 
\begin{align}
\label{eq: genWeights}
 \xx_{k+1} = w_{ii}\xx_k + \sum_{j}w_{ij}\xx[j]_k 
\end{align}
$W$ obeys the underlying graph, i.e., for $i\neq j$, $W_{ij} = 0$ if $(i,j)\notin\mE$. In other words, each node updates its value by taking a weighted sum of its own previous value with those of its neighbors. In short, the equation can be written as 
\begin{align}
 x_{k+1} = Wx_k
\end{align}
Such an algorithm is said to achieve consensus if 
\begin{align}
 \lim_{k\rightarrow\infty}\xx_k = r \triangleq \frac{1}{N}\sum_{j}\xx[j]_0
\end{align}
In such a static setup where the weights and the underlying interaction graph does not change with time, it is well known that consensus is achieved if and only if 
\begin{align}
\label{eq: oneoneT}
 \lim_{k\rightarrow\infty}W^k = \frac{1}{N}\ones\ones^T
\end{align}
Further \eqref{eq: oneoneT} holds if and only if the following conditions hold (e.g., see \cite{Xiao04})
\begin{enumerate} 
 \item $W$ is doubly stochastic, i.e., 
	\begin{align}
\label{eq: dSto}
        \ones^T W = \ones^T,\,\,\, W\ones = \ones
       \end{align}
 \item $\rho\bra{W - \frac{1}{N}\ones\ones^T} < 1$
\end{enumerate}
Note that $x_k = W^kx_0$. Under the above conditions, $x_k \rightarrow \frac{1}{N}\ones\ones^Tx_o = r\ones$. The convergence rate, $\crate(W)$, of the above consensus algorithm is formally defined as 
\begin{align}
 \crate(W) = \sup_{x_o\neq r\ones}\lim_{k\rightarrow\infty}\left[\frac{\Vert x_k - r\ones\Vert}{\Vert x_o - r\ones\Vert}\right]^{\frac{1}{k}}
\end{align}
and is given by $\crate(W) = \rho\bra{W - \frac{1}{N}\ones\ones^T}$. There is a considerable amount of work that explores different choices of $W$ and how it affects the rate of convergence of the consensus algorithm (e.g., \cite{Xiao04}). 

For the purpose of this paper and for ease of exposition, we use a specific but natural choice of $W$ (e.g., \cite{Murray04}) given by $W = I - \eps\L$, where $\L$ is the Laplacian of the interaction graph $\mG$, i.e., $\L = D - A$. $D = \text{diag}\{\Delta_i\}$ where $\Delta_i$ is the degree of node $i$. Let $0=\lambda_{N}(\L)\leq\lambda_{N-1}(\L)\leq\ldots\leq\lambda_1(\L)$ denote the eigen values of $\L$. The multiplicity of the zero eigen value is the number of connected components in the graph and $\lambda_{N-1}(\L) > 0$ if and only if the graph is connected. 

For such a choice of $W$, the spectral radius is given by $\rho(W - \frac{1}{N}\ones\ones^T) = \max\{1 - \eps\lambda_{N-1}(\L), \eps\lambda_1(\L)-1\}$. We state this as a Lemma for later reference.
\begin{lem}
 \label{lem: boyd}
The convergence rate, $\crate$, of \eqref{eq: genWeights} with $W = I-\eps\L$ is
\begin{align}
\label{eq: crateLap}
 \crate = \max\{1 - \eps\lambda_{N-1}(\L), \eps\lambda_1(\L)-1\}
\end{align}
\end{lem}
So, the conditions 1) and 2) above are satisfied if and only if $\eps < \frac{2}{\lambda_1(\L)}$. Furthermore, the convergence rate $\crate$ is maximized when the two quantities in \eqref{eq: crateLap} coincide, i.e., when
\begin{align}
\label{eq: crate}
\eps = \eps^{*} = \frac{2}{\lambda_1(\L) + \lambda_{N-1}(\L)} 
\end{align}
In particular, any $\eps < 1/\Delta$ will work where $\Delta = \max_i\Delta_i$. We remark that the techniques presented in the paper are independent of the choice of the weight matrix $W$. Whenever we wish to write closed form expressions for the convergence rates, we use the specific choice $W = I - \eps^*\L$ for simplicity.

\section{Communication Model}

In practice, the communication links between nodes can be unreliable. Conventionally, this has been taken into account by allowing the interaction topology to change with time. So, at time k, the connectivity between nodes is described by the graph $\mG_k$ where $\mG_k$ can now vary with time. There is a considerable amount of literature on the problem of achieving consensus under such time varying interaction topologies (\cite{Moreau05,Murray07, Ren05, Hatano05, Boyd06}). We model unreliable communication as packet erasures. So, at each time $k$, the packet transmitted from node $i$ to, say, node $j$ is either received ($X_k^{ji} = 1$) or erased ($X_k^{ji} = 0$). Similary, the packet sent from node $j$ to node $i$ is either received ($X_k^{ij}=1 $) or erased ($X_k^{ij}=0$). We consider two erasure models
\begin{enumerate}
 \item {\em Symmetric: } $X_{k}^{ij} = X_k^{ij}$, and $X_k^{ij}$, $X_k^{m\ell}$ are independent of each other whenever $(i,j) \notin \{(m,\ell),(\ell,m)\}$
 \item {\em Asymmetric: } $X_k^{ij}$, $X_k^{m\ell}$ are independent of each other whenever $(i,j) \neq (m,\ell)$, in particular $X_k^{ij}$ and $X_k^{ji}$ are independent. 
\end{enumerate}

The literature on consensus over time varying topologies only captures the symmetric case. Even though, consensus under very general conditions has been established, not much appears to be available by way of the rate of convergence. Under the asymmetric erasure model, the resulting interaction graph is effectively directed. An edge between node $i$ and $j$ is replaced by a pair of directed edges. The effective graph at any time depends on the packets that were erasured in that round. Under this setup, we define the adjacency matrix $A = [a_{ij}]$ and the Laplacian $\L$ as follows; $a_{ij} = 1$ if $(i\leftarrow j)\in\mE$ and $\L = D-A$ with $D = \text{diag}\{\Delta_i\}$ and $\Delta_i = \sum_j{a_{ij}}$. The resulting adjacency matrix and the Laplacian are not symmetric in general. As a result, they are not doubly stochastic either, i.e., $\ones^T\L\neq \ones^T$. When the graph $\mG$ is directed, (Olfati-Saber Murray 2007) prove that average consensus is achieved using a fixed $W = I - \eps\L$ if and only if the interaction graph $\mG$ is balanced, i.e., the in-degree of each node is equal to its out degree (cite Olfati-Saber Murray 2007). But when the link failures are random, the resulting interaction graph will generally not be balanced at every time step. But with coding, one can overcome this problem as we will show later. 
\section{Does Coding Help?}
\label{sec: codingHelps}

It turns out coding does help. In fact, to study the effect of coding we need to distinguish between the symmetric and asymmetric erasure models. When the erasures are symmetric, i.e., when $X_k^{ij} = X_k^{ji}$, this means that node $i$ (respectively, node $j$) {\em knows} what node $j$ (respectively, $i$) has received. For example, if node $i$ successfully received a packet from node $j$, it knows that node $j$ also successfully received the packet intended for it; alternately if node $i$ receives an erasure from node $j$, it knows that the packet intended for node $j$ was also erased. In this case, the links between the different nodes are erasure links with feedback (where the transmitter knows what the receiver receives). For erasure links with feedback it is well known that the optimal coding scheme is {\em retransmission}, i.e., the transmitter retransmits its packet until it is received at the receiver. 

When the erasures are not symmetric, one needs a more sophisticated coding scheme (called tree codes). We shall furher explain this below. 

When there are erasures and when there is no coding, an iteration of the consensus algorithm at node $i$ is given by
\begin{align}
\label{eq: noisy}
 \xx_{k+1} = \xx_k - \eps\sum_{j}a_{ij}X_k^{ij}(\xx_k - \xx[j]_k)
\end{align}
The effective adjacency matrix at time $k$ is then $A_k = A\circ X_k$, where $X_k = [X_k^{ij}]$. The associated Laplacian is $\L_k = D_k - A_k$ where $D_k^i = \sum_{j}A_k^{ij} = \sum_j a_{ij}X_k^{ij}$. 

\subsubsection{Symmetric Erasures} 
In this case, note that even without coding, the nodes achieve average consensus albeit at a slower rate depending on the erasure probability, say $p$. We show that coding (in this case retransmitting untill sucessful reception) results in faster convergence whenever there exists a constant $R' > 0$ such that 
\begin{align}
\label{eq: symmCodingGain}
 D(1-R',p) &> \log(\Delta + 1)\\
 \crate^{R'} &< \sqrt{\lambda_2(\Gamma)}
\end{align}
where $\crate$ is as in \eqref{eq: crateLap}, $H(.)$ is the binary entropy function and $\Gamma$ is defined in Lemma \ref{lem: symNC}. 

\subsubsection{Asymmetric Erasures}
Since $X_k^{ij}$ and $X_k^{ji}$ are independent, they are not equal in general. Note that $\L_k\ones = \ones$ but $\ones^T\L_k \neq \ones^T$ in general which violates \eqref{eq: dSto}. Furthermore, the associated graph is not balanced either, $\sum_ja_{ij}X_k^{ij} \neq \sum_ia_{ji}X_k^{ji}$, in general. In this case, the nodes will not achieve {\em average} consensus. But under very mild conditions, it is well known that the nodes achieve an agreement, i.e., $x_k \rightarrow Y\ones$ where $Y$ is a random variable that does not necessatily concentrate around the initial average $r$. But tree codes allow us to simulate the original recursions, i.e., \eqref{eq: genWeights}, and hence guarantee asymptotic average consensus. Before proceeding further, we provide a brief introduction to tree codes.
\section{Background on tree codes}
The problem of achieving consensus over erasure channels is an instance of the problem of simulating interactive communication protocols between a network of agents over unreliable links. In the specific case of consensus, the interactive communication protocol amounts to executing \eqref{eq: genWeights} at every node. In this context, Rajagopalan et al in \cite{Rajagopalan94} use {\em tree codes} to simulate such protocols with exponentially vanishing probability of error in the length of the protocol (e.g., the length of the protocol is said to be $m$ if one needs to execute $m$ iterations of \eqref{eq: genWeights}). Another very important instance of such interactive communication problems is one of stabilizing unstable dynamical systems over noisy communication channels (cite Sahai here). Even though the central role of tree codes in such problems has been identified, there have been no practical constructions until very recently. In \cite{Sukhavasi1101, Sukhavasi1103}, the authors proposed an explicit ensemble of {\em linear} tree codes with efficient decoding for the erasure channel. Equipped with this construction of tree codes, we can examine more closely how they can be used for specific problems such as consensus over erasure links which is what we do here. Before proceeding further, we will digress a little bit to outline the codes proposed in \cite{Sukhavasi1103} and list their relevant properties. 

\subsection{Linear time-invariant tree codes}
\label{subsec: treeCodes}
A tree code is essentially a semi-infinite causal encoding scheme which has a certain \lq Hamming distance'-like property. When decoding using maximum likelihood decoding over a discrete memoryless channel (DMC), such a tree code
guarantees exponentially small error probability with delay. In other words, the probability of incorrectly decoding a symbol (or paket) $d$ time steps in teh past decays exponentially in $d$. If the rate of the code is $R< 1$, such a causal encoding/decoding scheme with such an exponentially decaying probability of error (exponent $\beta$ say) is said to be $(R,\beta)-$anytime reliable. We will make this more precise below. We will describe the tree codes of (our work) in terms of their anytime reliability rather than in terms of their distance properties, because ultimately it is the exponent and rate that matter when communicating over DMCs. Since communication is packetized, let $\Lam$ denote the packet length. Each packet can be viewed as a symbol from $\F_2^\Lambda$. Suppose information is generated at the rate of $nR$ packets per time instant at the encoder. Then a rate $R$ time-invariant causal linear code is given by
\begin{align}
\label{eq: causalCode}
 c_t = G_1b_t + G_2b_{t-1} + \ldots + G_tb_1,\,\,\,t\geq 1
\end{align}
where $c_t \ \in \F_2^{n\Lam}$, $b_i \in \F_2^{nR\Lam}$ and $G_i \in \F_2^{n\Lam\times nR\Lam}$. So, at each time, the encoder receives $nR$ packets and transmits $n$ packets. Note that this is essentially a convolutional code with infinite memory. The decoder, at each time $t$, generates estimates $\hat{b}_{\tau|t}$ for $1\leq \tau \leq t$ where $\hat{b}_{\tau|t}$ denotes the decoder's estimate of $b_\tau$ using the channel outputs received till time $t$.
\begin{define}[Anytime Reliability]
 A causal code as in \eqref{eq: causalCode} is said to be $(R,\beta)-$anytime reliable if
\begin{align}
 P\bra{\hat{b}_{\tau|t} \neq b_\tau} \leq 2^{-\beta(t-\tau+1)},\,\,\,\forall\,\,\,\tau,t\geq d_o
\end{align}
for some fixed $d_o$ independent of $\tau,t$.
\end{define}
Let $p' = p^{1/\Lam}$. In \cite{Sukhavasi1103}, the authors showed that if the entries of $G_i$ are drawn i.i.d Bernoulli (1/2), then almost every code in this ensemble is $(R,\beta)-$anytime reliable for $R < 1-p'$ and $\beta < n\Lam E(R)$, where $E(R)$ is an exponent that depends on the DMC and that can be explicitly computed. For the packet erasure channel with erasure probability $p$, $E(R)$ is given by (see \cite{Sukhavasi1103})
\begin{align}
 E(R) = \left\{\begin{array}{ll}
               H^{-1}(1-R)\log\frac{1}{p'}, & R \leq \gamma_1\\
	       1-\log(1+p')-R,& \gamma_1 \leq R \leq \gamma_2\\
	       R\log\frac{R}{1-p'} + (1-R)\log\frac{1-R}{p'}, & \gamma_2 \leq R \leq 1-p'
              \end{array}
\right.
\end{align}
where 
\begin{align}
 \gamma_1 = 1-H\bra{\frac{p'}{1+p'}},\,\, \gamma_2 = \frac{1-p'}{1+p'}
\end{align}
For the rest of the analysis, we will assume that we are given an $(R,\beta)-$anytime reliable code with $d_o=0$.
\section{Main Results}
We present the results separately for the case of symmetric and asymmetric erasures.

\subsection{Symmetric Link Failures}
\label{sec: resultsSym}
Note that the underlying interaction graph $\mG$ is fixed while each link is modeled as a packet erasure channel. The graph $\mG$ is assumed to be connected and the links are undirected. If all agents know that link failures are symmetric, then each link is effectively a packet erasure channel with feedback. In each communication round, node $i$ would know that its packet transmission to node $j$ is erased if it receives an erasure from node $j$ in the same round. Recall that the consensus algorithm in the case where there are no erasures is given by 
\begin{align}
 x_{k+1} = (I - \eps\L)x_k
\end{align}
In particular, node $i$ performs the algorithm
\begin{align}
 \label{eq: noiseless}
\xx_{k+1} = \xx_{k} - \eps\sum_{j}a_{ij}(\xx_k - \xx[j]_k)
\end{align}
We now define the communication protocol.

\subsubsection{The Protocol}
A communication round is defined as one in which every node in the graph transmits one packet to each of its neighbors. The nodes are said to have completed $m$ iterations if all of them successfully computed $m$ iterations of \eqref{eq: noiseless}. Note that this will in general take more than $m$ communication rounds. Since each link is effectively an erasure channel with feedback, the optimal communication scheme at each node is to retransmit until successful reception. We describe this more precisely as follows. Let {\em e} denote an erasure. For each edge $j\rightarrow i$, we associate an input queue, $Q_{in}^{ij}$, and an output queue, $Q_{out}^{ij}$. $\Qin^{ij}$ contains the packets transmitted by node $j$ to node $i$ up to and including communication round $t$ while $\Qout^{ij}$ contains the packets received by node $i$ from node $j$. 

Also let $b_t^{ij}$ denote the packet transmitted by node $j$ to node $i$ in communication round $t$ and let $z_t^{ij}$ denote the received packet. Then
\begin{align}
 z_t^{ij} = \left\{\begin{array}{cl}
                    b_t^{ij} & \text{w.p }1-p\\
		    \textit{e} & \text{w.p } p
                   \end{array}
\right.
\end{align}
Now if $z_t^{ji} = \textit{e}$, then node $j$ infers that $b_t^{ij}$ was erased and hence retransmits it in the next communication round unless $b_t^{ij}$ was a \lq wait' symbol which we describe as follows. We say that a node $i$ has \lq new data' if it could compute one or more new iterations of \eqref{eq: noiseless}. During communication rounds where node $j$ does not have any new data to transmit, it transmits a wait symbol which we denote with {\em w}. The transmission from node $i$ to node $j$ in round $t$ is described in Algorithm \ref{alg: repetition}. Let $\mN_i$ denote the neighbors of node $i$, i.e., $\mN_i = \{j'| a_{ij'} = 1\}$.

\begin{algorithm}                      
\caption{Node $i$'s transmission to node $j$ in round $t$}          
\label{alg: repetition}                           
\begin{algorithmic}[1]                   
    \IF{$z_{t-1}^{ji} = \ite$ \AND $b_{t-1}^{ji} \neq \itw $}
        \STATE $b^{ji}_t = b^{ji}_{t-1}$, i.e., re-transmit
    \ELSE
        \STATE For each $j'\in \mN_i$, let $\ell_{t,j'} = \max\{\ell'\mid\xx[j']_{\ell'}\in \Qout^{ij'}\}$
	\STATE Compute $\ell_t = \min_{j'\in \mN_i}\ell_{t,j'}$
	\IF{$\ell_t = \ell_{t-1}+1$}
		\STATE Compute $\xx_{\ell_t+1}$ using \eqref{eq: noiseless} and set $b_t^{ji} = \xx_{\ell_t+1}$ (note that $\ell_t \leq \ell_{t-1}+1$)
	\ELSE
		\STATE i.e., $\ell_t = \ell_{t-1}$, set $b_t^{ji} = \itw$
	\ENDIF
    \ENDIF
\end{algorithmic}
\end{algorithm}
The algorithm is illustrated through an example in Fig \ref{tab: repetition}. Using such an algorithm, we have the following bounds on the convergence rate of average consensus.
\begin{figure}
\begin{center}
\begin{tabular}{cc|c|c|c|c|}
\cline{2-6}
$\Qout[5]^{i,1}$ : & $\xx[1]_0$ & $\xx[1]_1$ & \ite       & \ite    & $\xx[1]_2$ \\
\cline{2-6}
$\Qout[5]^{i,2}$ : & $\xx[2]_0$ & \ite       & $\xx[2]_1$ & \itw    & $\xx[2]_2$ \\
\cline{2-6}
$\Qin[5]^{1,i}$ :  & $\xx_0$    & $\xx_1$    & \itw       & $\xx_2$ & $\xx_2$ \\
\cline{2-6}
$\Qin[5]^{2,i}$ :  & $\xx_0$    & $\xx_1$    & $\xx_1$    & $\xx_2$ & \itw \\
\cline{2-6}
& \multicolumn{5}{c}{past $\longleftarrow$ $\longrightarrow$ present}
\end{tabular}
\caption{Consider an instance of the queues at node $i$. Suppose its only neightbors are nodes 1 and 2. In round 2, node $i$ receives an erasure from node 2 and infers that its own transmission to node 2 must also have been erased. As a result, node $i$ re-transmits $\xx_1$ to node 2 in round 3. Similarly in round 3, node $i$ knows that its transmission to node 1 was erased. Since the erased symbol was only a \lq wait', node $i$ does not re-transmit it in round 4. Instead, it checks if it can perform another iteration of \eqref{eq: noiseless}. In this case, it can and hence transmits the new data $\xx_2$ to node 1. In round 5, node $i$ does not have any new data to transmit to node 2 and hence transmits a \lq wait'.}
\label{tab: repetition}
\end{center}
\end{figure}

\begin{thm}
\label{thm: sym}
 Let $P_{M,R'}$ denote the probability that the network requires more than $M$ communication rounds to compute $MR'$ iterations of \eqref{eq: noiseless}. Further suppose that the packet erasure probability is $p$ and that erasures are symmetric. Then
\begin{align}
 P_{M,R'} \leq N2^{-M\bra{D(1-R',p) - \log(\Delta+1)}}
\end{align}
In particular, whenever $R'$ satisfies 
\begin{align}
\label{eq: RprimeS}
 D(1-R',p) > \log(\Delta + 1)
\end{align}
$P_{M,R'}$ decays exponentially fast in $M$. 
Recall that $N$ is the number of nodes and $\Delta$ the maximum degree. 
\end{thm}
\begin{proof}
See Appendix \ref{subsec: sym}.
\end{proof}
Using Theorem \ref{thm: sym}, we can determine the convergence rate Algorithm \ref{alg: repetition}, $\crate_c^s$, and it is given by
\begin{align}
 \label{eq: crateSymCode}
\crate_c^s \leq \crate^{R'}
\end{align}
where $R'$ is the largest rate such that \eqref{eq: RprimeS} is satisfied and $\crate$ is defined in \eqref{eq: crateLap}. The superscript and subscript in $\crate_c^s$ denote that it is the convergence rate with coding under symmetric erasures. We will compare this with the convergence rate without coding in Section \ref{sec: discussion}. Let 
\begin{align}
R(p) \triangleq \sup_{R'\geq 0}\{R'\mid D(1-R',p) > \log(\Delta + 1)\} 
\end{align}
Then it is easy to see that $R(p) > 0$ if and only if $p < 1/(1+\Delta)$. This means that the proof technique used here does not allow us to prove average consensus if the erasure proability is larger than $1/(1+\Delta)$. We can demonstrate how to overcome this. In fact, one can show that average consensus will be acheived for all $0\leq p\leq 1$, we will state the result as follows.
\begin{thm}
\label{thm: symImprov}
 Let $P_{M,R'}$ denote the probability that the network requires more than $M$ communication rounds to compute $MR'$ iterations of \eqref{eq: noiseless}. Further suppose that the packet erasure probability is $p$ and that erasures are symmetric. Then
\begin{align}
 P_{M,R'} \leq N2^{-MD(R',(1-p)^{|\mE|})}
\end{align}
In particular, whenever $R'$ satisfies 
\begin{align}
\label{eq: RprimeSimprov}
 R' < (1-p)^{|\mE|}
\end{align}
$P_{M,R'}$ decays exponentially fast in $M$. 
Recall that $N$ is the number of nodes and $|\mE|$ is the number of edges in the network.
\end{thm}
\begin{proof}
 See Appendix \ref{subsec: symImprov}
\end{proof}
Combining Theorems \ref{thm: sym} and \ref{thm: symImprov}, we conclude that the the convergence rate of Algorithm \ref{alg: repetition}, $\crate_c^s$, is given by
\begin{align}
 \crate_c^s \leq \min\{\crate^{R(p)}, \crate^{(1-p)^{|\mE|}}\}
\end{align}

\subsection{Asymmetric Link Failures and Tree Codes}

Now suppose packet erasures are not symmetric. Since information at each node is generated one packet at a time and since the unit of communication is a packet, the rate of the code is $R = 1/n$\footnote{This kind of rate is because we are quantizing each number $\xx_k$ to fit into one packet. One can instead quantize it more finely into multiple packets, say $k$, in which case $R = k/n$}. Here, one round of communication corresponds to every pair of neighbors exchanging $n$ packets each. Then in any communication round, node $i$ does not known which of the $n$ transmitted packets have been received by each of its neighbors. In this case, we use the anytime reliable codes described in Section \ref{subsec: treeCodes}.  
\subsubsection{The protocol}
Consider the pair of nodes $i,j$ and let $b_t^{ji}$ denote the $t^{th}$ information packet destined to node $j$ from node $i$. Then the data actually transmitted by node $i$ is given by
\begin{align}
 c_\ell^{ji} = \sum_{\ell'=1}^\ell G_{\ell'} b_{\ell'}^{ji}
\end{align}
Since the code is $(R,\beta)-$anytime reliable, we have $P(\hat{b}_{\ell'|\ell}^{ji}\neq b_{\ell'}^{ji}) \leq 2^{-\beta(\ell-\ell')}$. Since the channel is an erasure channel, the maximum likelihood decoder amounts to solving linear equations. This can be done recursively and efficiently as shown in (our paper). Whenever the equations admit a unique solution to some of the variables, those variables are correctly decoded. We leave the remaining variables as erasures and do not venture a guess about their value. As a result, the decoder always knows whenever it decodes something correctly. 

Like in the case of repetition coding for symmetric erasures, for each link $j\rightarrow i$, we associate two queues $\Qin^{ij}$ and $\Qout^{ij}$ although with a slightly different meaning. The queue $\Qin^{ij}$ contains all the information packets transmitted by node $j$ to node $i$ till round $t$. In other words, $\Qin^{ij} = \{b_\tau^{ij}\}_{\tau\leq t}$. On the other hand, $\Qout^{ij}$ are node $i$'s estimates of the information packets transmitted by node $j$ so far, i.e., $\Qout^{ij} = \{\hat{b}_{\tau|t}^{ij}\}_{\tau\leq t}$. Also, it will be evident from Algorithm \ref{alg: treeCode} that $\Qin^{ji} = \Qin^{j'i}$ for all $j,j'\in \mN_i$.

With this setup, the mechanics of the protocol is very simple and is outlined in Algorithm \ref{alg: treeCode}
\begin{algorithm}
 \caption{Node $i$'s transmission to its neighbors in round $t$}
 \label{alg: treeCode}
\begin{algorithmic}[1]
  \STATE For each $j'\in \mN_i$, compute $\ell_{t,j'} = \max\{\ell'\mid \xx[j']_{\ell'}\in \Qout^{ij'}\}$ and let $\ell_t = \min_{j'\in \mN_i}\ell_{t,j'}$
  \STATE Also compute $m_{t,j'} = \max\{m'\mid \xx_{m'}\in \Qin^{j'i}\}$ and let $m_t = \min_{j'\in \mN_i}m_{t,j'}$
  \IF{$\ell_t + 1 > m_{t-1}$}
	\STATE Compute $\xx_{m_{t-1}+1}$ using \eqref{eq: noiseless} and set $b_t^{ji} = \xx_{m_{t-1}+1}$ for all $j\in \mN_i$
  \ELSE 
	\STATE set $b_t^{ji} = \itw$ for all $j\in \mN_i$
  \ENDIF
\end{algorithmic}
\end{algorithm}

We can now compute the convergence rate of average consensus achieved by the above algorithm and we state it as the following Theorem.
\begin{thm}
\label{thm: asym}
 Let $P_{M,R'}$ denote the probability that the network requires more than $M$ communication rounds to compute $MR'$ iterations of \eqref{eq: noiseless}. Further suppose that the packet erasure probability is $p$ and that erasures are asymmetric. Suppose each node uses a $(R,\beta)-$anytime reliable code. Then
\begin{align}
 P_{M,R'} \leq N2^{-M\bra{(1-R')\frac{\beta}{2} - H(R') - \log(\Delta+1)}}
\end{align}
In particular, whenever $R'$ satisfies 
\begin{align}
\label{eq: RprimeA}
 (1-R')\beta/2 > H(R') + \log(\Delta + 1)
\end{align}
$P_{M,R'}$ decays exponentially fast in $M$. 
\end{thm}
\begin{proof}
See Appendix \ref{subsec: asym}
\end{proof}
As in the symmetric case, the convergence rate, $\crate_c^a$, using tree codes is given by
\begin{align}
 \label{eq: crateAsymCode}
\crate_c^a = \crate^{R'}
\end{align}
where $R'$ is the largest rate such that \eqref{eq: RprimeA} is satisfied and $\crate$ is as in \eqref{eq: crateLap}.
Let 
\begin{align}
R(\beta) \triangleq \sup_{R'\geq 0}\left\{R'\mid (1-R')\beta/2 > H(R') + \log(\Delta + 1)\right\} 
\end{align}
Then much like in Section \ref{sec: resultsSym}, it is easy to see that $R(\beta) > 0$ if and only if $\beta > 2\log(1+\Delta)$. 

\section{Discussion - Coding Vs No Coding}
\label{sec: discussion}
When there is no coding, the consensus recursion is given by \eqref{eq: noisy}. We begin with the case of symmetric erasures.

\subsection{Symmetric Erasures}

The convergence rate of \eqref{eq: noisy} when erasures are symmetric is given by the following Lemma
\begin{lem}[Symmetric Erasures]
\label{lem: symNC}
When the erasures are symmetric and i.i.d over time and space, the convergence rate of \eqref{eq: noisy}, $\crate_{\cbar}^s$ which we define as
\begin{align}
\crate_{\cbar}^s = \sup_{x_o\neq r\ones}\lim_{k\rightarrow\infty}\left[\frac{\E\Vert x_k - r\ones\Vert^2}{\Vert x_o - r\ones\Vert^2}\right]^{\frac{1}{2k}}
\end{align}
is given by
\begin{align}
 \crate_{\cbar}^s = \sqrt{\lambda_2(\Gamma_s)}
\end{align}
where $\Gamma_s = \E(I - \eps\L_0)\otimes (I - \eps\L_0)$ is a deterministic matrix that is a function of $\eps, p, \L$ and can be computed explicitly in closed form. The subscript $\cbar$ indicates that there is no coding and the subscript $s$ in $\Gamma_s$ is because the erasures are symmetric
\end{lem}
\begin{proof}
See Appendix \ref{proof: symNC}.
\end{proof}

Consider the case of coding in the presence of symmetric erasures. From Theorem \ref{thm: sym} and \eqref{eq: crate}, it is easy to see that the convergence rate is given by $\crate_c^s$ in \eqref{eq: crateSymCode}. So, whenever $\crate_c^s < \crate_{\cbar}^s$, coding offers an advantage. We state this as a Theorem
\begin{thm}
\label{thm: crateSym}
 In the case of symmetric erasures, coding offers a faster convergence than \eqref{eq: noisy} whenever there is a $R' > 0$ such that 
\begin{subequations}
\label{eq: symDiscuss}
\begin{align}
 (1-R')\log\frac{1}{p} &> \log(\Delta + 1) + H(R') \\
 \bra{\rho(I - \eps\L - \frac{1}{N}\ones\ones^T)}^{R'} &< \sqrt{\lambda_2(\Gamma_s)}
\end{align}
\end{subequations}
\end{thm}

\subsection{Asymmetric Erasures}
As mentioned in Section \ref{sec: codingHelps}, when link failures are asymmetric, the algorithm of \eqref{eq: noisy} does not achieve average consensus. Nevertheless the nodes reach agreement and the rate of convergence to agreement has been characterized in \cite{Zhou2009}. Here, we characterize the mean squared error of the state from average consensus.
\begin{lem}[Asymmetric Erasures]
\label{lem: asymNC}
When the erasures are asymmetric and i.i.d over time and space, we have
\begin{align}
\label{eq: asymNC}
\E\Vert x_k - r\ones\Vert^2 = (x_o-r\ones)^T\otimes(x_o-r\ones)^T\Gamma_a^k vec(I)
\end{align}
Here $I$ is an $N\times N$ identity matrix and
\begin{align}
\Gamma_a = \E(I - \eps\L^T_0)\otimes (I - \eps\L^T_0)
\end{align}
where $\Gamma_a$ is a deterministic matrix that is a function of $\eps, p, \L$ and can be computed explicitly in closed form. Furthermore $\rho(\Gamma_a) = 1$.
\end{lem}
\begin{proof}
See Appendix \ref{proof: asymNC}.
\end{proof}
Note that $\ones^T\Gamma_a = \ones^T$ but $\Gamma_a\ones \neq \ones$. Let $c$, $\Vert c\Vert = 1$ be the right eigen vector of $\Gamma_a$ corresponding to eigen value 1, i.e., $\Gamma_a c = c$. Then, it is easy to see that $\lim_{k\rightarrow\infty}\Gamma_a^k = \frac{1}{N}c\ones^T$. Using this in \eqref{eq: asymNC}, we get
\begin{align}
 \lim_{k\rightarrow\infty}\E\Vert x_k - r\ones\Vert^2 = (x_o-r\ones)^T\otimes(x_o-r\ones)^Tc
\end{align}
This proves that one cannot achieve average consensus without coding when link failures are asymmetric. So, a major benefit of using tree codes in such cases is to {\em guarantee} average consensus. Furthermore, tree codes can be used to implement {\em any} distributed protocol over a network with erasure links.
\appendix

\subsection{Proof of Lemma \ref{lem: symNC}}
\label{proof: symNC}

Note that $\L_k\ones = 0$ whether or not the erasures are symmetric. Recall that $r = \frac{1}{N}\ones^Tx_0$.
\begin{subequations}
\begin{align}
 & x_{k} - r\ones = (I - \eps\L_{k-1})(x_{k-1} - r\ones)\\
& x_{k} - r\ones  = \\
 & \underbrace{(I - \eps\L_{k-1})(I - \eps\L_{k-2})\ldots(I - \eps\L_{0})}_{\triangleq Y_k}(x_{0} - r\ones)
\end{align}
\end{subequations}
\begin{align}
\label{eq: symm1}\E\Vert x_{k} - r\ones\Vert^2 &= (x_{0} - r\ones)^T\E Y^T_kY_k (x_{0} - r\ones)\nonumber\\
&= (x_{0} - r\ones)^T\otimes(x_{0} - r\ones)^T vec(P_k)
\end{align}
where $P_k = \E Y^T_kY_k$. Recall that the erasure process is independent over time and across links. Then we have
\begin{subequations}
\label{eq: symm2}
\begin{align}
P_k &= \E(I - \eps\L^T_{0})P_{k-1}(I - \eps\L_{0})\\
 vec(P_k) &= \Gamma_s vec(P_{k-1}),\,\,\,\text{where}\\
 \label{eq: gammasym}\Gamma_s &= \E(I - \eps\L^T_0)\otimes (I - \eps\L^T_0)
\end{align}
\end{subequations}
Since erasures are symmetric, $\L^T_0 = \L_0$. Furthermore, we have $vec(P_k) = \Gamma_s^k vec(I)$, where $I$ is an $N\times N$ identity matrix. Putting \eqref{eq: symm1} and \eqref{eq: symm2} together, we get
\begin{align}
 \E\Vert x_{k} - r\ones\Vert^2 = (x_{0} - r\ones)^T\otimes(x_{0} - r\ones)^T\Gamma_s^k vec(I)
\end{align}
So, the rate of convergence of the consensus algorithm in the absence of coding is clearly determined by $\Gamma_s$. Observe that $\Gamma_s$ is doubly stochastic, i.e., $\ones^T\Gamma_s = \ones^T$ and $\Gamma_s\ones = \ones$. It has one eigen value at 1 and all others are strickly smaller than 1 in magnitude. Let $\lambda_2(\Gamma_s)$ denote the second largest eigen value in magnitude. Then clearly
\begin{align}
 \lim_{k\rightarrow\infty}\Gamma_s^k = \frac{1}{N^2}\ones\ones^T
\end{align}
and the rate of convergence is given by
\begin{align}
 \crate_{\cbar}^s = \sqrt{\lambda_2(\Gamma_s)}
\end{align}

\subsection{Proof of Lemma \ref{lem: asymNC}}
\label{proof: asymNC}

Except the claim $\rho(\Gamma_a) = 1$, everything else follows from Appendix \ref{proof: symNC}. Since $\Gamma_a = \E(I - \eps\L^T_0)\otimes (I - \eps\L^T_0)$, the claim $\rho(\Gamma_a) = 1$ follows if $\rho(I - \eps\L_0) = 1$ which is what we show. Recall that the random variable $X^{ij}_0$ is defined as $X^{ij}_0 = 0$ if the link $j\rightarrow i$ is erased at time $0$ and $X^{ij}_0 = 1$ otherwise. For brevity, we will write $X^{ij}$ instead of $X^{ij}_0$. Then it is easy to verify that one can write $\L_0$ as follows
\begin{align}
 \L_0 = \sum a_{ij}X^{ij}e_i(e_i-e_j)^T
\end{align}
where $e_i$ is the $i^{th}$ unit vector. In particular, the underlying Laplacian in the absence of any erasures can be written as $\L = \sum a_{ij}e_i(e_i-e_j)^T$. For any $x\in\Re^N$, we have
\begin{align}
 x^T(I - \eps\L_0)x &= x^T\bra{I - \frac{\eps}{2}(\L_0 + \L^T_0)}x\nonumber\\
\label{eq: asymNC1}&= \Vert x\Vert^2 - \frac{\eps}{2}\sum a_{ij}X^{ij}(x_i-x_j)^2 \leq \Vert x\Vert^2
\end{align}
Furthermore,
\begin{align}
 \Vert x\Vert^2 - \frac{\eps}{2}\sum a_{ij}X^{ij}(x_i-x_j)^2 &\geq \Vert x\Vert^2 - \frac{\eps}{2}\sum a_{ij}(x_i-x_j)^2\nonumber\\
\label{eq: asymNC2}&= x^T(I-\eps\L)x \geq -\Vert x\Vert^2
\end{align}
The last inequality follows from the fact that $\rho(I-\eps\L) = 1$. Combining \eqref{eq: asymNC1} and \eqref{eq: asymNC2}, we have $|x^T(I-\eps\L_0)x| \leq \Vert x\Vert^2$ for all $x\in\Re^N$ which implies that $\rho(I-\eps\L_0)\leq 1$. But $\L_0\ones = \ones$, so $\rho(I-\eps\L_0) = 1$. Therefore $\rho(\Gamma_a) = 1$. This completes the proof.

\subsection{Proof of Theorem \ref{thm: sym}}
\label{subsec: sym}
We will begin by identifying the state of the protocol in Algorithm \ref{alg: repetition}. For the sake of clarity, we will refer to nodes using letters $u,v$, etc., instead of $i,j$. Recall that $\mN_v$ denotes the set of neighbors of $v$. For each node $v$ at time $t$ (i.e., after round $t$), we associate $|\mN_v|$ variables $\{n_{vu}(t)\}_{u\in \mN_v}$, where $n_{vu}(t)$ denotes the latest iterate of node $u$ that is available to node $v$ at time $t$. In other words, $n_{vu}(t)$ is the largest integer $\tau$ such that $\xx[u]_{\tau}$ is available to node $v$. We further define
\begin{align}
 \label{eq: defnu}
n_v(t) \triangleq 1 + \min_{u\in \mN_v}n_{vu}(t)
\end{align}
Note that $n_v(t)$ is the latest iteration of \eqref{eq: noiseless} that node $v$ can compute at time $t$. In other words, node $v$ has computed $\{\xx[v]_{\tau}\}_{\tau \leq n_v(t)}$ and no more. With this setup, it is clear that Algorithm \ref{alg: repetition} would have executed $\min_{v}n_v(t)$ iterations of \eqref{eq: noiseless} till time $t$. Note that the rate of the protocol is then given by $R = \lim_{t\rightarrow \infty}\frac{min_{v}n_v(t)}{t}$, which is a random variable for a specific run of the protocol. We now state the evolution of $n_{vu}(t)$ as a Lemma below.

\begin{lem}
 \label{lem: stateEvol}
Let $X^{vu}_t = 1$ if the edge $(v,u)$ is erased in round $t$ and $0$ otherwise. Then the evolution of $n_{vu}(t)$ is given by the following equation
\begin{align}
 \label{eq: stateEvol}
n_{vu}(t+1) = n_{vu}(t) + X^{vu}_{t+1}\indicator{n_u(t) > n_{vu}(t)}
\end{align}
\end{lem}
\begin{proof}
The proof follows from the following simple observations
\begin{enumerate}
 \item $n_{vu}(t)$ increases by atmost $1$ in each step
 \item In any round, if node $u$ receives an erasure on a link, it will infer that its transmission on that link was also erased. As a result, node $u$ has knowledge of $n_{vu}(t)$ at all times $t$
 \item In round $t+1$, if either the edge $(v,u)$ is erased or node $u$ sends a $\itw$ to node $v$, then $n_{vu}(t+1) = n_{vu}(t)$ 
 \item Node $u$ sends a \lq wait' $\itw$ to node $v$ in round $t+1$ if and only if $n_{vu}(t) = n_u(t)$.
\end{enumerate}. 
\end{proof}

We say that round $t$ got wasted at node $v$ if $n_v(t-1) = n_v(t)$, i.e., node $v$ could not perform a new iteration of \eqref{eq: noiseless} at time $t$. The proof idea is as follows: for each node $v$ at time t, we will argue that there exists a sequence of $t$ edges of which at least $t - n_v(t)$ edges have failed. We then union bound over all possible choices of such $t$ edges. 

Before proceeding further, we define an object which we call the \lq trellis', for lack of a better word. Associated to any undirected graph $\mG = (\mV,\mE)$ represented by the adjacency matrix $A$, we define an infinite trellis $\mT({\mG}) = (\mV_{\mT},\mathcal{E}_{\mT})$ as follows. Associated to each node $v$ in $\mV$, there are countably infinitely many copies $\{_k\}_{k\geq 0}$ in $\mV_\mT$. Let $I$ denote a $|\mV|\times|\mV|$ identity matrix. Then the nodes $\mV_\mT$ and edges $\mE_\mT$ of $\mT(\mG)$ are given by
\begin{subequations}
 \label{eq: trellis}
\begin{align}
 \label{eq: verts}\mV_\mT &= \bigcup_{v\in\mV}\bigcup_{k\geq 0}\{v_k\}\\
 \label{eq: links}\mE_\mT &= \left\{(v_\tau,u_{\tau'})\mid |\tau-\tau'| = 1, \bra{A+I}_{vu} = 1\right\}
\end{align}
\end{subequations}
The edges in $\mE_\mT$ are all undirected, i.e., $(u_0,v_1)$ and $(v_1,u_0)$ are treated as a single edge. The trellis for an example network is given in Fig \ref{fig: graphTrellisTrim}. 

\begin{define}[time-like]
\label{def: timelike}
Any sequence of edges (or a path), $\mS_t$, in the trellis $\mT(\mG)$ of the type
\begin{align*}
\mS_t = \left\{(v_t,u^{(t-1)}_{t-1}),(u^{(t-1)}_{t-1},u^{(t-2)}_{t-2}),\ldots,(u^{(1)}_1, u^{(0)}_0)\right\}
\end{align*}
will be called \lq time-like' ending in node $v_t$ 
\end{define}
An edge $(u^{(\tau)}_\tau, u^{(\tau-1)}_{\tau-1}) \in \mE_\mT$ is said to be erased if there was an erasure on the edge $(u^{(\tau)},u^{(\tau-1)}) \in \mE$ in round $\tau$. The time-like sequence $\mS_t$ is said to have $\ell$ erasures if $\ell$ of the $t$ edges in $\mS_t$ were erased. We are now ready to state the key Lemma from which the proof of Theorem \ref{thm: sym} follows easily. 
\begin{lem}
\label{lem: timeLike}
 If after $t$ rounds of communication, node $v$ has performed $n_{v}(t)$ iterations of \eqref{eq: noiseless}, then there exists a time-like sequence of $t$ edges ending in node $v_t$ that have at least $t - n_{v}(t)$ erasures among them.   
\end{lem}
We will first prove Theorem \ref{thm: sym} using Lemma \ref{lem: timeLike}. Suppose after $t$ communication rounds, node $v$ performed $Rt$ iterations of \eqref{eq: noiseless}, for some $R < 1-p$. Recall that the probability of an erasure is $p$. Then there must be a time like sequence of $t$ edges with at least $(1-R)t$ erasures, the probability of which is approximately $2^{-tD(1-R, p)}$, where $D(q,p) = q\log(q/p) + (1-q)log(1-q/1-p)$. Now there are at most $(\Delta+1)^t$ choices of such time-like sequences. Then, doing a union bound over all these sequences, we get
\begin{align}
 P_{R,t} \leq N(\Delta+1)^t2^{-tD(1-R,p)}
\end{align}
where $P_{R,t}$ is the probability that the network performed $Rt$ or fewer iterations of \eqref{eq: noiseless} in $t$ rounds and $N$ is the number of nodes in the network. This is the claim in Theorem \ref{thm: sym}. We will now prove the Lemma.

\begin{proof}[Proof of Lemma \ref{lem: timeLike}]

For ease of presentation, we will introduce the following notation in the rest of the proof.
\begin{enumerate}[a)]
 \item {\em we will refer to any time-like sequence of $\tau$ edges ending in $v_\tau$ that has $\tau-n_v(\tau)$ or more erasures as a \lq\lq witness'' at $v_\tau$.}
 \item {\em We will call a node $u\in\mN_v$ a \lq\lq bottleneck'' for node $v$ in round $t$ iff $n_{vu}(t-1) = n_v(t-1)-1$, i.e., $n_{vu}(t-1) = min_{u'\in\mN_v}n_{vu'}(t-1)$.}
\end{enumerate}

The Lemma claims that there is a witness at $v_t$ for all $v\in\mV$ and $t\geq 0$. We will prove this by induction. The hypothesis is clearly true for $t=0$. Suppose it is true for all nodes $v\in\mV$ and all $\tau\leq t-1$. Recall that we say that round $t$ at node $v$ is wasted only if $n_v(t-1) = n_v(t)$. There are two broad cases, round $t$ gets wasted at node $v$ or it does not.

\begin{enumerate}[1)]
 \item Suppose round $t$ is not wasted, i.e., $n_v(t) = n_v(t-1)+1$. Then by the induction hypothesis, there is a witness at $v_{t-1}$. Appending the edge $(v_{t-1},v_t)$ to this witness gives us a witness for $v_t$.
 \item It remains to consider the case where round $t$ gets wasted at node $v$, i.e., $n_v(t) = n_v(t-1)$. 
\end{enumerate}

We will divide case 2) above into two sub-cases: a) $\exists$ a $u\in\mN_v$ s.t $n_u(t-1) = n_v(t-1)-1$ and b) such a neighbor does not exist.
\begin{enumerate}[a)]
 \item If there is a neighbor $u\in\mN_v$ such that $n_u(t-1) = n_v(t-1)-1$, then the witness for $v_t$ is obtained by appending the edge $(v_t,u_{t-1})$ to the witness at $u_{t-1}$. 
 \item Here $n_{u}(t-1) \geq n_v(t-1)$ for all $u\in\mN_v$. Since $|n_u(\tau) - n_v(\tau)| \leq 1$ for any $\tau$, we can partition the neighbors of $v$ into two classes $Y = \{u\in\mN_v\mid n_u(t-1)=n_v(t-1)\}$ and $Z = \{u\in\mN_v\mid n_u(t-1) = n_v(t-1)+1\}$. Furthermore, let $B = \{u\in\mN_v\mid n_{vu}(t-1) = n_v(t-1)-1\}$ denote the bottlenecks for $v$ in round $t$. 
\end{enumerate}

We will further divide case b) above into two sub-cases: i) $B\cap Z = \emptyset$ and ii) $B\cap Z \neq \emptyset$
\begin{enumerate}[i)]
 \item $B\cap Z = \emptyset$, i.e., there are no bottlenecks in the set of neighbors $Z$. Observe that a bottleneck neighbor will not send a wait \itw. Also for any $u\in B\cap Y$, $n_{vu}(t-1) = n_v(t-1)-1 = n_u(t-1)-1$. So, the data transmitted by node $u$ to node $v$ in round $t$ is $\xx[u]_{n_u(t)}$, i.e., iteration $n_u(t)$ of \eqref{eq: noiseless}. Since round $t$ at node $v$ got wasted, at least one of the edges to a bottleneck neighbor must have been erased in round $t$. Otherwise, node $v$ would have been able to compute a new iteration of \eqref{eq: noiseless} and the round would not have been wasted. Suppose the erasure happened on edge $(v,u)$ for some $u\in B\cap Y$. Then appending edge $(v_t,u_{t-1})$ to the witness at $u_{t-1}$ will give us the witness at $v_t$.
 \item $B\cap Z \neq \emptyset$, i.e., there is a neighbor $u\in B\cap Z$ such that $n_{u}(t-1) = n_v(t-1)+1$ and $n_{vu}(t-1) = n_v(t-1)-1 = n_u(t-1)-2$. Furthermore, there must be a neighbor $u \in B\cap Z$ whose transmission to $v$ in round $t$ must have been erased (else there must be an edge to $B\cap Y$ which was erased and we revert back to case i)). Note that $n_u(t-2) \geq n_v(t-1)$. It follows from Lemma \ref{lem: stateEvol} that node $u$ must have transmitted iteration $n_v(t-1)$ in round $t-1$ as well as round $t$ and both were erased since $n_{vu}(t) = n_{vu}(t-1) = n_v(t-1)-1$. Since this erasure model considers symmetric erasures, the transmission from $v$ to $u$ in round $t-2$ is also erased. Appending the edges $(v_t,u_{t-1})$ and $(u_{t-1},v_{t-2})$ to the witness at $v_{t-2}$ gives us the witness for $v_t$.  
\end{enumerate}
This completes the proof of Lemma \ref{lem: timeLike}.

\end{proof}

\begin{figure}
\centering
\subfigure[An example network]{\includegraphics[scale = 0.278]{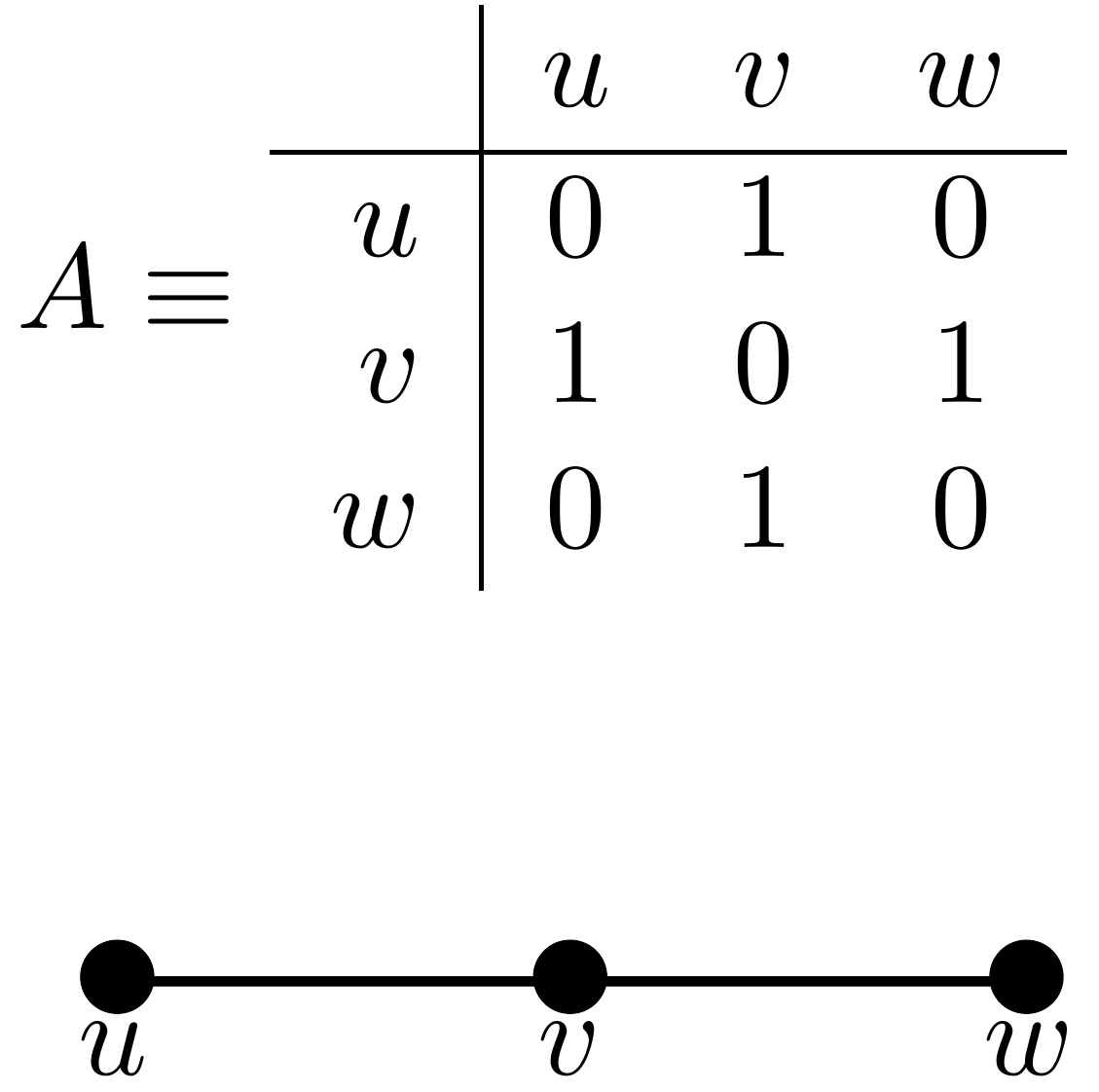}
\label{fig: graphTrim}
}
\subfigure[Trellis associated to the network in (a)]{\includegraphics[scale = 0.278]{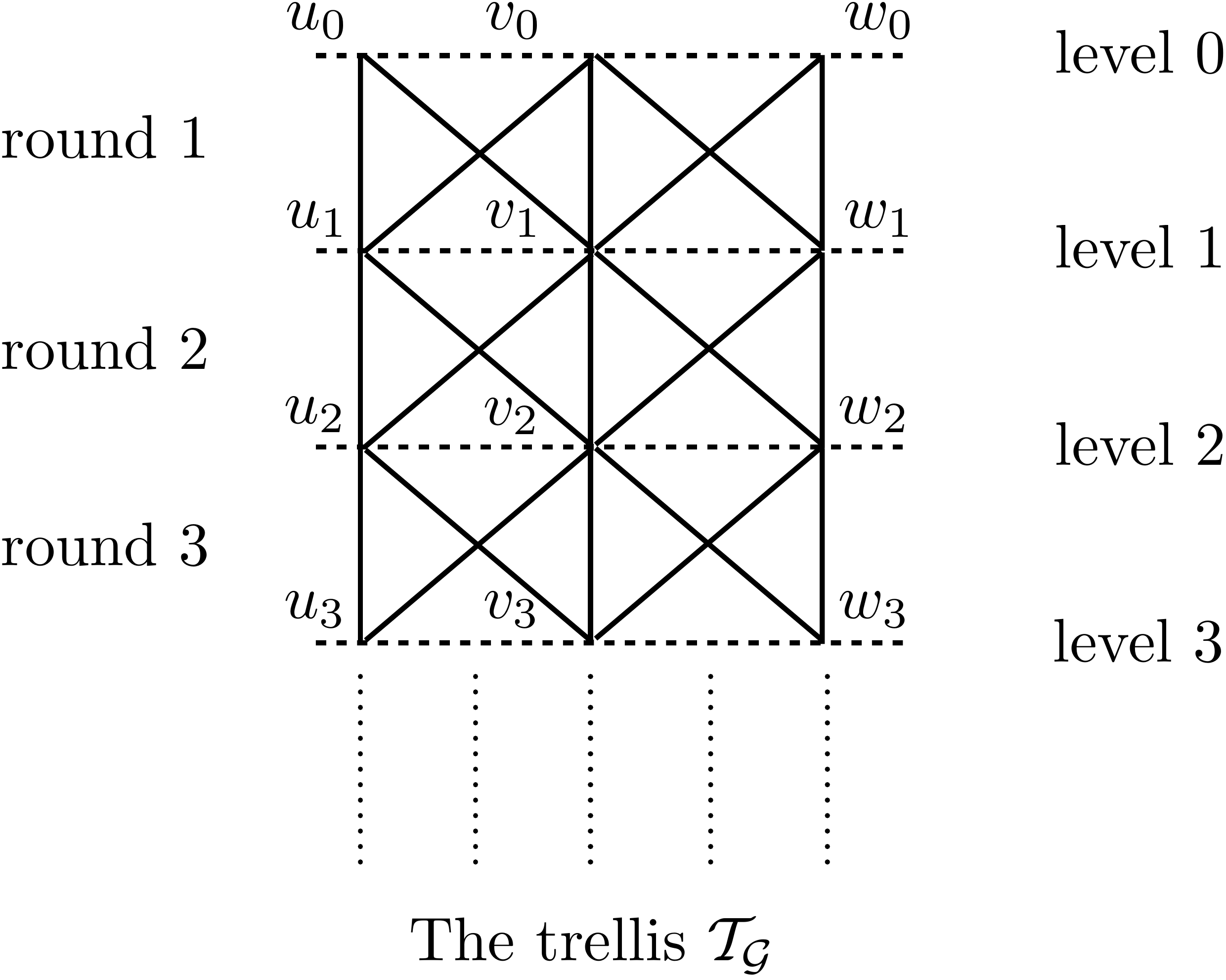}
\label{fig: trellisTrim}
}
\caption{This depicts the trellis associated to a network of three nodes connected in a straight line. The thick lines represent edges.}
\label{fig: graphTrellisTrim}
\end{figure}

\subsection{Proof of Theorem \ref{thm: asym}}
\label{subsec: asym}
We will begin the proof with three preliminary results before moving to the main argument. Recall that an $(R,\beta)-$anytime reliable code is one that guarantees $P\bra{\hat{b}_{\tau|t}\neq b_\tau} \leq 2^{-\beta(t-\tau+1)}$. For such a code that is linear, we can say the following.
\begin{lem}
 \label{lem: disjointIntervals}
Suppose $\{b_i\}_{i\geq 0}$ are encoded and decoded using a causal linear $(R,\beta)-$anytime reliable code. Consider the following events, $Y(\tau_1',\tau_1)$: $\tau_1' = 1 + \argmax_{\ell}\{\hat{b}_{\ell|\tau_1} = b_\ell\}$ and $Y(\tau_2',\tau_2)$: $\tau_2' = 1 + \argmax_{\ell}\{\hat{b}_{\ell|\tau_2} = b_\ell\}$, i.e., $Y(\tau_i',\tau_i)$ is the event that at decoding instant $\tau_i$, the position of the earliest error is at $\tau_i'$ for $i=1,2$. Furthermore, suppose that the intervals $[\tau_1',\tau_1]$ and $[\tau_2', \tau_2]$ are disjoint. Then we have
\begin{align}
\label{eq: disjointIntervals} P\bra{Y(\tau_1',\tau_1)\cap Y(\tau_2',\tau_2)} \leq 2^{-\beta\bra{|\tau_1-\tau_1'+1|+|\tau_2 - \tau_2'+1|}}
\end{align} 
The probability above is only over the randomness of the channel.
\end{lem}
\begin{proof}
 Without loss of generality, assume that $\tau_2' > \tau_1$. Due to linearity, we can assume without losing generality that the input $b_i=0$ for $i\geq 0$. Let $E_i$ denote the portion of the erasure pattern introduced by the channel during the interval $[\tau_i',\tau_i]$ that resulted in the event $Y(\tau_i',\tau_i)$. Then, we claim that $P(E_i) \leq 2^{-\beta|\tau_i-\tau_i'+1|}$. This follows from the simple observation that if the encoder input in the first $\tau_i-\tau_i' + 1$ instants is all zero and the corresponding channel erasure pattern is $E_i$, then $Y(\tau_i',\tau_i)$ implies that at the decoding instant $\tau_i-\tau_i'$, the earliest error would have happened at time $0$, the probability of which is at most $2^{-\beta|\tau_i-\tau_i'+1|}$. 

Since the intervals $[\tau_1',\tau_1]$ and $[\tau_2', \tau_2]$ are disjoint, the erasure patterns $E_1$ and $E_2$ correspond to independent channel uses. So we have
\begin{align*}
 P\bra{Y(\tau_1',\tau_1)\cap Y(\tau_2',\tau_2)} \leq P(E_1,E_2) = P(E_1)P(E_2) 
\end{align*}
 The result now follows.
\end{proof}

For ease of presentation, we introduce the following definition
\begin{define}[Error Interval]
 With respect to the notation in Lemma \ref{lem: disjointIntervals}, we refer to the interval $[\tau_i',\tau_i]$ as the error at time $\tau_i$. 
\end{define}

Before proceeding with the rest of the proof, we will recall a Lemma from \cite{Schulman96} and state it here for easy reference. 
\begin{lem}[Lemma 7, \cite{Schulman96}]
\label{lem: Schulman96}
In any finite set of intervals on the real line whose union $J$ is of total length $s$ there is a subset of disjoint intervals whose union is of total length at least $s/2$
\end{lem}

We will now state a version of Lemma \ref{lem: disjointIntervals} when the error intervals are not necessarily disjoint.
\begin{lem}
 \label{lem: intervalThingy}
If $\{b_i\}_{i\geq 0}$ are encoded and decoded using a causal linear $(R,\beta)-$anytime reliable code, then
\begin{align*}
 P\bra{\hat{b}_{\tau_1'|\tau_1}\neq b_{\tau_1'},\ldots,\hat{b}_{\tau_m'|\tau_m}\neq b_{\tau_m'}} \leq 2^{-\frac{\beta(\sum_i |\tau_i - \tau_i' + 1|)}{2}}
\end{align*}
\end{lem}
\begin{proof}
 The proof follows directly from Lemma \ref{lem: disjointIntervals} and Lemma \ref{lem: Schulman96}.
\end{proof}

We use an argument very similar to the one used in proving Theorem \ref{thm: sym}. We will define a trellis $\vec{\mT}(\mG)$ exactly the same way we defined $\mT(\mG)$ except that the edges $\vec{\mE}_\mT$ are now directed and they point forward in time, i.e., downwards w.r.t to the Fig \ref{fig: trellisTrim}. In other words, for neighbors $(u,v)\in\mV$, the edge $(v_t,u_{t-1})$ is directed from node $u_{t-1}$ to node $v_t$ and represents the transmission from $u$ to $v$ in round $t$. 

Recall the definition of a time-like sequence of edges, $\mS_t$, from Definition \ref{def: timelike}. Let 
\begin{align*}
 \mS_t = \left\{(u^{(t)}_t,u^{(t-1)}_{t-1}),(u^{(t-1)}_{t-1},u^{(t-2)}_{t-2}),\ldots,(u^{(1)}_1, u^{(0)}_0)\right\}
\end{align*}
Let $B_{\tau}$ be the error interval at decoding instant $\tau$ on the edge node $(u^{(\tau)},u^{(\tau-1)}) \in \mE$. We alternately call $B_\tau$ the error interval on the edge $(u^{(\tau)}_\tau,u^{(\tau-1)}_{\tau-1}) \in \vec{\mE}_\mT$. Then we define $|\mS_t|$ as follows
\begin{align}
 \label{eq: errInterval}
|\mS_t| &= \sum_{(v,u)\in\mE}|B_{vu}|, \text{ where} \\
B_{vu} &= \bigcup_{\tau: (u^{(\tau)},u^{(\tau-1)}) = (v,u)} B_{\tau}
\end{align}
This definition is motivated by the fact that the packet erasure events during an error interval on a given edge, say $(v,u)\in\mE$, are independent of those in an error interval on a different edge $(v',u')\neq (v,u)$ in any round of communication. So, intuitively $|\mS_t|$ captures the number of independent \lq\lq bad'' channel realizations seen by the edges in $\mS_t$. In what follows, we will show a connection between the number of wasted communication rounds at the node $u^{t}$ and the number $|\mS_t|$.

A {\em witness} at node $v_t$ is a time like sequence of edges $\mS_t$ such that $|\mS| \geq t - n_v(t)$. In Lemma \ref{lem: witness}, we will demonstrate a witness for $v_t$ for all $v\in\mV$ and $t \geq 0$. The technique is very similar to the proof of Lemma \ref{lem: timeLike} and hence we will only provide a sketch of the proof. After that we will use Lemma \ref{lem: intervalThingy} to prove that $P(t-n_v(t) \geq m) \leq (\Delta+1)^t\binom{t}{m}2^{-m\beta/2}$ for any $v\in\mV$.

\begin{lem}
 \label{lem: witness}
 If after $t$ rounds of communication, node $v$ has performed $n_{v}(t)$ iterations of \eqref{eq: noiseless}, then there exists a time-like sequence, $\mS_t$ of $t$ edges in $\vec{\mE}_\mT$ ending in node $v_t$ with $|\mS_t| > t - n_{v}(t)$
\end{lem}
\begin{proof}
 The proof is obtained by repeating the same argument as in the proof of Lemma \ref{lem: timeLike} with the word \lq erasure' replaced with the word \lq tree code error'. The only case that needs a little bit of clarification is case 2-b-ii, i.e., round $t$ is wasted at node $v$ and $B\cap Z \neq\emptyset$, where $B$ and $Z$ retain the same meaning as before. In this case, like before, there is a neighbor $u \in \mN_v$ such that $n_{vu}(t) = n_u(t-1) - 2$. From Algorithm \ref{alg: treeCode}, it is clear that node the information $\xx[u]_{n_u(t-1)-1}$ was encoded and transmitted by node $u$ to node $v$ in round $t-1$ or before. Therefore, the error interval on the edge $(v_t, u_{t-1}) \in \vec{\mE}_\mT$ contains the interval $[t-1,t]$. Let the witness at node $u_{t-1}$ be $\mS_{t-1,u}$.
Append the edge $(v_t, u_{t-1})$ to $\mS_{t-1,u}$ to get a new time-like sequence which we call $\mS_{t,v}$. We claim that $\mS_{t,v}$ is a witness at $v_t$. This proof of this claim follows from the following observations
\begin{enumerate}
 \item When applying Lemma \ref{lem: intervalThingy}, we only to care about error intervals on the same edge at different times
 \item The edge $(v,u)$ appears in the time-like sequence $\mS_{t,v}$ for round $t$ and hence, it can possibly appear again only in $\mS_{t,v}$ in round $t-2$ or earlier. So, the length of the union of the error intervals on the edge $(v_\tau,u_{\tau-1})\in\mS_{t-1,u}$ increases by at least 2 with the addition of the edge $(v_t,u_{t-1})$. Hence we have
\begin{align*}
 |\mS_{t,v}| \geq |\mS_{t-1,u}| + 2 \geq t-1 - n_u(t-1) = t - n_v(t)
\end{align*}
This completes the proof of Lemma \ref{lem: witness}.
\end{enumerate}
Putting together Lemma \ref{lem: witness} and Lemma \ref{lem: intervalThingy}, we have
\begin{align*}
 P(t-n_v(t) \geq m) \leq (\Delta+1)^t\binom{t}{m}2^{-\beta m/2}
\end{align*}
The result now follows trivially.
\end{proof}

\subsection{Proof of Theorem \ref{thm: symImprov}}
\label{subsec: symImprov}

The bound $(1-p)^{|\mE|}$ is intuitively motivated by the following observation, in a given round of communication, $(1-p)^{|\mE|}$ is the probability that none of the edges are erased. As a result one would expect the fraction of communication rounds in which nodes can perform an iteration of \eqref{eq: noiseless} to be approximately $(1-p)^{|\mE|}$. The above observation alone would not render a proof because successful communication could also mean that a node received only \lq waits' from its neighbors and hence could not compute an iteration of \eqref{eq: noiseless}. The proof idea is simple but conveying it requires some setup. Let $\itW_{uv}^{(t)}$ denote the event where node $v$ transmits a \lq wait' to node $u$ in round $t$. We introduce the following definition
\begin{define}
\label{def: causeWait}
 Consider nodes $v$, $u$, $u'$ such that $u \in \mN_v$ and $u'\in\mN_u$. Also suppose that node $v$ transmits a \lq wait' to node $u$ in round $\tau$ and node $u$ transmits a \lq wait' to node $u'$ in round $\tau+1$, i.e., events $\itW_{uv}^{(\tau)}$ and $\itW_{u'u}^{(\tau+1)}$ happen. {\em Then $\itW_{uv}^{(\tau)}$ is said to have caused $\itW_{u'u}^{(\tau+1)}$} if both the following conditions hold
\begin{enumerate}[(a)]
 \item $n_{u}(\tau-1) = 1 + n_{uv}(\tau-1)$
 \item $n_{u'u}(\tau) = n_u(\tau)$
\end{enumerate}
\end{define}
To understand the definition, observe that condition (a) implies that node $v$ is a bottleneck node for node $u$ in round $\tau$ and condition (b) implies that node $u'$ already knows $n_u(\tau)$ after round $\tau$. Node $u$ could not perform a new iteration in round $\tau$ since it received a \lq wait' from a bottleneck node (in this case $v$) and hence sent a \lq wait' to node $u'$. So, it is natural to blame $\itW_{uv}^{(\tau)}$ for $\itW_{u'u}^{(\tau+1)}$. 
Note that Definition \ref{def: causeWait} is further justified by the observation that a \lq wait' in round $\tau$ will either have an effect in round $\tau+1$ or will never. Also note that Definition \ref{def: causeWait} can be extended to more than two waits by having conditions (a) and (b) hold for every pair of successive \lq wait' events.

With that, we are now ready to state the main Lemma. The Lemma essentially implies that \lq waits' do not loop in the network. In other words, if in round $\tau$ a node $v$ transmits a \lq wait', then this \lq wait' will not {\em cause} the same node $v$ to transmit another \lq wait' in a future round $\tau'>\tau$.

\begin{lem}[\lq Waits' do not loop]
\label{lem: noLooping}
 Consider the sequence of events $\{\itW_{u_{i+1}u_i}^{(\tau+i-1)}\}_{i=1}^{\ell}$ such that $\itW_{u_{i+1}u_i}^{(\tau+i-1)}$ is caused by $\itW_{u_{i}u_{i-1}}^{(\tau+i-2)}$ for all $2\leq i\leq \ell$. Then the nodes $\{u_i\}_{i=1}^\ell$ are all distinct. 
\end{lem}
\begin{proof}
 Node $u_1$ sent a \lq wait' to node $u_2$ in round $\tau$ implies that $n_{u_1}(\tau-1) = n_{u_2u_1}(\tau-1)$. Furthermore, since $\itW_{u_3u_2}^{(\tau+1)}$ is caused by $\itW_{u_2u_1}^{(\tau)}$, conditions (a) and (b) in Definition \ref{def: causeWait} apply. In particular, condition (a) together with the first observation gives $n_{u_2}(\tau-1) = n_{u_1}(\tau-1)+1$. Since node $u_2$ could not perform a new iteration of \eqref{eq: noiseless}, we have $n_{u_2}(\tau) = n_{u_2}(\tau-1) = n_{u_1}(\tau-1)+1$. Repeating this argument for the remaining nodes, we get 
\begin{align}
\label{eq: contra}
 n_{u_{i+1}}(\tau+i-1) = n_{u_i}(\tau_i-2),\,\, \forall\,\, 1\leq i\leq \ell
\end{align} 

Now suppose the nodes $\{u_i\}_{i=1}^\ell$ are not all distinct. In particular, suppose $u_\ell = u_1$. Then from \eqref{eq: contra}, we have $n_{u_1}(\tau+\ell-2) = n_{u_\ell}(\tau+\ell-2) = \ell - 1 + n_{u_1}(\tau-1)$ which is not possible since $n_{u_1}(.)$ can increment by atmost 1 in each round and $n_{u_1}(\tau) = n_{u_1}(\tau-1)$. 

One will similarly arrive at a contradiction if any other node repeats in $\{u_i\}_{i=1}^\ell$. 
\end{proof}
The implication of Lemma \ref{lem: noLooping} is clear. If a node $v$ sends a \lq wait' in round $\tau$ to any of its neighbors, then this \lq wait' will not by itself stop node $v$ from performing an iteration of \eqref{eq: noiseless} in a future round.

We are now ready to provide the main argument. Let $d(v,u)$ denote the length of the shortest path from node $u$ to node $v$. So, if $v\in\mN_u$, then $d(v,u) = 1$ and $d(v,v) = 0$. Let the diameter of the graph be $\delta$, i.e., $\delta = \max_{u,v\in\mV}d(v,u)$. And for an edge $e_{uu'}\equiv (u,u')\in\mE$, we define 
\begin{align*}
 d(v,e_{uu'}) = \min\{d(v,u),d(v,u')\}
\end{align*}
Let $\mE_v^{(i)} \triangleq \{e\in\mE\mid d(v,e) = i\}$. In view of Lemma \ref{lem: noLooping}, it is not difficult to see that an erasure on an edge in $\mE_v^{(i)}$ in round $\tau$ will have an effect (if any) at node $v$ only in round $\tau+i$. Let $A_{i,\tau}$ denote the event that there is an erasure on an edge in $\mE_v^{(i)}$ in round $\tau$. Then for $\tau \geq \delta$, it is easy to see that $\cap_{i=0}^{\delta}A_{i,\tau-i}^c$ implies that the round $\tau$ at node $v$ is {\em not} wasted, i.e., node $v$ can compute an iteration of \eqref{eq: noiseless}. In other words
\begin{align*}
 P(n_v(\tau) = n_v(\tau-1)) \leq 1-P\bra{\bigcap_{i=0}^{\delta}A_{i,\tau-i}^c} = 1 - (1-p)^{|\mE|}
\end{align*}
Due to the erasure model, note that the even $A_{i,\tau}$ is independent of $A_{i',\tau'}$ for $(i,\tau)\neq (i',\tau')$. Let
\begin{align*}
 X_\tau &= \indicator{n_v(\tau) = n_v(\tau-1)}\\
 Y_\tau &= \indicator{\cup_{i=0}^{\delta}A_{i,\tau-i}}
\end{align*}
Then from the above argument $X_\tau = 1$ implies $Y_\tau = 1$ and $\{Y_\tau\}$ are independent Bernoulli random variables. Note that $P(Y_\tau = 1) \leq 1 - (1-p)^{|\mE|}$. Let $R' = \frac{n_v(t)}{t}$, then we have
\begin{align*}
 P(t - n_v(t) = m) &= P\bra{\sum_{\tau=0}^tX_\tau = m} \leq P\bra{\sum_{\tau=0}^t Y_\tau \geq m} \leq 2^{-tD(1-R',1-(1-p)^{|\mE|})} = 2^{-tD(R',(1-p)^{|\mE|})}
\end{align*}
The last inequality follows from a standard Chernoff bounding technique and is true whenever $R' < (1-p)^{|\mE|}$. Union bounding over all nodes $v\in\mV$, we have
\begin{align*}
 P\bra{\exists\,\,v\in\mV\,\,\ni n_v(t)\leq R't} \leq N2^{-tD(R',(1-p)^{|\mE|})}
\end{align*}
This completes the proof.

\bibliographystyle{IEEEbib}
\bibliography{CDC2012arXiv}

\end{document}